\documentclass[11pt]{amsart}
\usepackage[T1]{fontenc}
\usepackage{newtxtext,newtxmath}
\usepackage[a4paper,textwidth=155mm,textheight=235mm]{geometry}
\usepackage{mathtools,microtype,flafter,needspace}
\usepackage{tikz}
\usetikzlibrary{arrows.meta,calc}
\usepackage{enumitem}
\usepackage[hidelinks]{hyperref}
\usepackage{bookmark}
\usepackage{etoolbox}
\hypersetup{pdftitle={From log 2 to pi/2: the sharp asymptotic inradius of polynomial lemniscates},pdfauthor={Yujian Geng and Dong Qiu}}
\newtheorem{theorem}{Theorem}[section]
\newtheorem{lemma}[theorem]{Lemma}
\newtheorem{proposition}[theorem]{Proposition}
\newtheorem{corollary}[theorem]{Corollary}
\theoremstyle{remark}

\numberwithin{equation}{section}
\setlist[enumerate]{itemsep=3pt,topsep=5pt}
\allowdisplaybreaks[1]
\newcommand{\C}{\mathbb C}\newcommand{\R}{\mathbb R}
\newcommand{\D}{\mathbb D}
\newcommand{\Pn}{\mathcal P_n}\newcommand{\Om}{\Omega}
\DeclareMathOperator{\arsinh}{arsinh}
\title[The sharp asymptotic inradius of polynomial lemniscates]{From $\log 2$ to $\pi/2$: the sharp asymptotic inradius of polynomial lemniscates}
\author{Yujian Geng}
\author[Dong Qiu]{Dong Qiu\textsuperscript{*}}
\makeatletter
\patchcmd{\@setauthors}{\endtrivlist}{%
  \par\nobreak\vspace{4pt}%
  \normalfont\footnotesize\itshape
  School of Mathematics \& Center for Applied Mathematics of Guangxi (Guangxi University),\\
  Guangxi University, Nanning, Guangxi, 530004, P.\ R.\ China\par
  \endtrivlist
}{}{\PackageError{paper}{Unable to place the author affiliation}{Check the author title block.}}
\makeatother
\thanks{\textsuperscript{*}Corresponding author: Dong Qiu. E-mail: \href{mailto:qiudong@gxu.edu.cn}{\texttt{qiudong@gxu.edu.cn}}.}
\date{}
\begin{document}
\begin{abstract}
Let $R_n$ be the infimum of the inradii of $\{z:|p(z)|<1\}$ over monic degree-$n$ polynomials whose zeros lie in the closed unit disk. We prove $nR_n\to\pi/2$, matching the asymptotic obstruction supplied by $z^n-1$. We first establish the exact universal radius $2^{1/n}-1$ for disks centered at zeros, which recovers the $(\log2)/n$ bound. Small inradius then forces radial concentration of the zeros and decay of their low reciprocal moments. These estimates give an entire limit with a modulus reflection identity; a second rescaling produces an exponential tangent and strict sublevel disks of every radius below $\pi/2$. The proof is accompanied by a Lean~4 formalization and a step-by-step source index.
\end{abstract}
\maketitle

\noindent\textbf{Formal verification.}
The main theorem and its supporting formal proof have been verified in Lean~4. The mathematical conventions, source correspondence, versions, and reproduction procedure are specified in the Declarations at the end of the paper. The accompanying step index records the relation between the written argument and the formal proof.

\section{Introduction}
For $n\ge1$, let
\[
 \Pn=\left\{p(z)=\prod_{j=1}^n(z-z_j):z_j\in\overline\D\right\},
 \qquad \D=\{z\in\C:|z|<1\}.
\]
Zeros are counted with multiplicity. For $p\in\Pn$, define
\begin{align}
 \Om_p&=\{z\in\C:|p(z)|<1\},\notag\\
 \rho(p)&=\sup\{r>0:\exists a\in\C,\ B(a,r)\subset\Om_p\},
 \qquad R_n=\inf_{p\in\Pn}\rho(p).
 \label{eq:definitions}
\end{align}
Here $B(a,r)$ denotes the open Euclidean disk with center $a$ and radius $r$. The center is unrestricted. We use the strict sublevel set throughout and do not assume that either extremum in~\eqref{eq:definitions} is attained.

\emph{Historical context and attribution.}
Erd\H{o}s, Herzog, and Piranian posed the asymptotic inradius problem in 1958~\cite[Problem~3, p.~134]{EHP}. In the same passage, they used $p(z)=z^n-1$ to identify $\pi/2$ as an upper obstruction for the coefficient of the inverse-degree scale. Thus the candidate numerical constant and its model example precede the present work. Their question should be distinguished from an assertion that this coefficient is attained: the matching lower estimate is the content of Theorem~\ref{thm:main}. Later estimates were obtained by Pommerenke~\cite{Pom} and by Krishnapur, Lundberg, and Ramachandran~\cite{KLR}; the latter also discuss the roots-of-unity configuration as the conjectural extremal model.

The more recent inverse-degree lower bound was announced by Price in the public discussion of Problem~1039~\cite{Price,EP}. Sothanaphan explained the product argument giving the coefficient $\log2$~\cite{Sothanaphan}, and Bloom stated the matched-product inequality used below~\cite{Bloom}. Kitamura supplied a public Lean formalization of the baseline argument~\cite{Kitamura}. These public contributions are cited in their original form; they are not treated as journal publications. We give the product proof in Section~\ref{sec:baseline} and develop the asymptotic lower bound in Sections~\ref{sec:radial}--\ref{sec:conclusion}.

\Needspace{7\baselineskip}
Our main result determines the sharp asymptotic coefficient.
\begin{theorem}\label{thm:main}
For the extremal inradius defined in~\eqref{eq:definitions},
\begin{equation}\label{eq:main}
 \lim_{n\to\infty}nR_n=\frac\pi2.
\end{equation}
More precisely, for every $\varepsilon>0$ there is $N_\varepsilon$ such that
\[
 \rho(p)\ge\frac{\pi/2-\varepsilon}{n}
 \qquad(n\ge N_\varepsilon,\ p\in\Pn).
\]
The coefficient $\pi/2$ is optimal.
\end{theorem}

The coefficient $\pi/2$ is optimal in the statement with an arbitrary loss $\varepsilon>0$ and a degree threshold depending on that loss. This is an asymptotic assertion: it does not assert $R_n\ge\pi/(2n)$ at every finite degree; indeed, $R_1=1$.

The proof begins with a finite-product argument. If every disk centered at a zero fails to lie in $\Om_p$, one selects a point with $|p|\ge1$ from each disk and estimates the product of the resulting values. This gives the exact universal radius
\[
 T_n=2^{1/n}-1,\qquad nT_n\longrightarrow\log2,
\]
when the center must be a zero. Indeed, put $x=1/n$ and use
\[
 \frac{e^{(\log2)x}-1}{x}\longrightarrow\log2
 \qquad(x\downarrow0),
\]
which follows from the derivative of the exponential at zero. To determine the unrestricted asymptotic inradius, we instead study sequences satisfying $\rho(p_n)\le C/n$ for a fixed $C>0$.

The small-inradius assumption ensures that every slightly larger disk contains a point where $|p_n|\ge1$. Local Harnack contraction turns these points into exponential bounds on the radial defect and the low reciprocal moments. The normalized logarithmic derivative has a uniform real-part bound. After its small low-order terms have been removed, a high-order Schwarz estimate controls the remaining analytic logarithm on circles at distance $O(1/n)$ from the unit circle.

Two rescalings complete the argument. The first, in the coordinate $z=\xi_ne^{w/n}$, yields an entire function with uniform left asymptotics and an exact modulus reflection identity. The second translates and normalizes this function far to the left. Three-lines convexity identifies a pure exponential tangent with exponent at most $1$. The positivity of $\cos y$ on $|y|<\pi/2$ then supplies the disks needed for the sharp lower bound.

\section{The product bound and disks centered at zeros}\label{sec:baseline}
We follow the kernel-product method in the public baseline proof~\cite{Sothanaphan,Kitamura} to establish the matched-product inequality stated by Bloom~\cite{Bloom}. We include the details both for completeness and to isolate the precise restriction imposed by fixing the center at a zero.

\begin{lemma}[Kernel product]\label{lem:kernel}
If $|u_j|=1$ for $1\le j\le n$ and $0\le a<1$, then
\begin{equation}\label{eq:kernel}
 \prod_{i,j=1}^n|1-a\overline{u_i}u_j|\le(1-a^n)^n.
\end{equation}
\end{lemma}
\begin{proof}
First take $|v_j|<1$ and write
\[
 B(w)=\prod_j\frac{w-v_j}{1-\overline{v_j}w}
     =\sum_{m\ge0}b_mw^m,\qquad \lambda=B(0)=\prod_j(-v_j).
\]
We use the real Hilbert space $\ell^2(\mathbb N;\mathbb C)$, with inner product
$\langle u,v\rangle=\Re\sum_m\overline{u_m}v_m$.
Here $\sum_m|b_m|^2=1$. To verify this directly in coefficient space, multiplication by one Blaschke factor $(w-a)/(1-\bar a w)$ sends a square-summable sequence $u$ to a sequence $v$ satisfying, with $d=1-|a|^2$,
\[
 t_0=u_0,\quad t_{m+1}=u_{m+1}+\bar a t_m,
 \qquad v_0=-au_0,\quad v_{m+1}=-au_{m+1}+dt_m.
\]
Expansion and telescoping give
\[
 \sum_{m=0}^N|v_m|^2+d|t_N|^2=\sum_{m=0}^N|u_m|^2.
\]
The last term tends to zero: for a finitely supported sequence this follows from $|a|<1$, and the same identity bounds the error after truncating an arbitrary square-summable sequence. Thus multiplication by one factor preserves the squared norm. Iterating from the constant sequence $(1,0,\ldots)$ proves the assertion about $b$; the evaluated series represents $B$ by the Cauchy product formula.

Set $h_m=\mathbf1_{m=0}-\bar\lambda b_m$, $S_m=\sum_jv_j^m$, and $H_m=\overline{h_m}$. These sequences belong to $\ell^2$. The geometric series, evaluation at each $v_j$, and $B(v_j)=0$ give
\[
 \langle S,H\rangle=n,\qquad
 \|H\|_2^2=1-|\lambda|^2,\qquad
 \|S\|_2^2=\sum_{i,j}\frac1{1-\overline{v_i}v_j}.
\]
For the middle identity, expand $\sum|h_m|^2$, using $b_0=\lambda$ and $\sum|b_m|^2=1$. The sum in the last identity is real by conjugate symmetry. Real Hilbert-space Cauchy--Schwarz now gives
\begin{equation}\label{eq:kernel-sum}
 \sum_{i,j}\frac1{1-\overline{v_i}v_j}
 \ge\frac{n^2}{1-\prod_j|v_j|^2}.
\end{equation}
Repeated points cause no difficulty.

Define $L(a)=\sum_{i,j}\log|1-a\overline{u_i}u_j|$. All factors are nonzero. Differentiating for $a>0$ and pairing conjugate terms gives
\[
 aL'(a)=n^2-\sum_{i,j}\frac1{1-a\overline{u_i}u_j}.
\]
Apply~\eqref{eq:kernel-sum} with $v_j=\sqrt a\,u_j$ to obtain
\[
 L'(a)\le-\frac{n^2a^{n-1}}{1-a^n}
 =\frac d{da}\bigl[n\log(1-a^n)\bigr].
\]
The difference $L(a)-n\log(1-a^n)$ is therefore nonincreasing and has value zero at $a=0$. Exponentiation proves the claim; the endpoint $a=0$ is immediate.
\end{proof}

\begin{lemma}[Regularized separation product]\label{lem:separation}
Suppose that $|u_j|\le1$, $\eta>0$, and $\gamma=\arsinh(\eta/2)$. Then
\begin{equation}\label{eq:separation}
 \prod_{i<j}(|u_i-u_j|^2+\eta^2)
 \le\left(\frac{\sinh(n\gamma)}{\sinh\gamma}\right)^n.
\end{equation}
\end{lemma}
\begin{proof}
The maximum modulus principle and Lemma~\ref{lem:kernel} imply
\begin{equation}\label{eq:offdiag}
 \prod_{i\ne j}|1-a\overline{u_i}u_j|
 \le\left(\frac{1-a^n}{1-a}\right)^n\qquad(0\le a<1).
\end{equation}
Indeed, with the other variables fixed, the factors involving $u_k$ have product
$|\prod_{j\ne k}(1-a\overline{u_j}u_k)|^2$.
Moving the variables successively to the unit circle cannot decrease the product. On the circle the diagonal factors have product $(1-a)^n$, so~\eqref{eq:offdiag} follows by dividing~\eqref{eq:kernel} by this quantity.

Set $\lambda=e^\gamma$ and $a=\lambda^{-2}$, so that $\eta=\lambda-\lambda^{-1}$. Writing $D_i=1-|u_i|^2$, direct expansion yields
\[
 \lambda^2|1-a\overline{u_i}u_j|^2
 -(|u_i-u_j|^2+\eta^2)
 =(1-a)(D_i+D_j)+aD_iD_j\ge0.
\]
Multiply over $i<j$ and apply~\eqref{eq:offdiag}. The identity
\[
 \lambda^{n-1}\frac{1-a^n}{1-a}
 =\frac{\sinh(n\gamma)}{\sinh\gamma}
\]
then gives~\eqref{eq:separation}.
\end{proof}

By scaling, if $|y_i|\le A$ and $A,\tau>0$, then
\begin{equation}\label{eq:scaled-separation}
 \prod_{i<j}(|y_i-y_j|^2+\tau^2)
 \le A^{n(n-1)}
 \left(\frac{\sinh(n\gamma)}{\sinh\gamma}\right)^n,
 \qquad \gamma=\arsinh\frac{\tau}{2A}.
\end{equation}

\begin{proposition}[Matched-product inequality]\label{prop:matched}
If $p(z)=\prod_i(z-z_i)\in\Pn$ and $|w_i-z_i|\le t$ for $t\ge0$, then
\begin{equation}\label{eq:matched}
 \prod_i|p(w_i)|\le\bigl((1+t)^n-1\bigr)^n.
\end{equation}
\end{proposition}
\begin{proof}
The case $t=0$ is immediate. Assume $t>0$ and put $A=1+t$, so that $|w_i|\le A$. The identity
\[
 (w_i-z_j)(w_j-z_i)
 =(w_i-z_i)(w_j-z_j)-(z_i-z_j)(w_i-w_j)
\]
implies
\[
 |w_i-z_j|\,|w_j-z_i|
 \le |z_i-z_j|\,|w_i-w_j|+t^2.
\]
For $X,Y\ge0$, the two-dimensional Cauchy--Schwarz inequality gives
\[
 (XY+t^2)^2\le(X^2+t^2/A)(Y^2+t^2A).
\]
Separate the diagonal factors and pair the off-diagonal factors indexed by $i<j$. This gives
\[
 \left(\prod_i|p(w_i)|\right)^2
 \le t^{2n}
 \prod_{i<j}(|z_i-z_j|^2+t^2/A)
 \prod_{i<j}(|w_i-w_j|^2+t^2A).
\]
We use the algebraic form of the separation estimate. If $|y_i|\le A_0$,
$\lambda\ge1$, $a=\lambda^{-2}<1$, and $\tau=A_0(\lambda-\lambda^{-1})$, then the factor comparison in Lemma~\ref{lem:separation} gives
\[
 \prod_{i<j}(|y_i-y_j|^2+\tau^2)
 \le(A_0^2\lambda^2)^{n(n-1)/2}
       \left(\frac{1-a^n}{1-a}\right)^n.
\]
Put $B=\sqrt A$, $a=A^{-1}$, and $J=(1-a^n)/(1-a)$. Then
\[
 t/B=B-B^{-1},\qquad tB=A(B-B^{-1}).
\]
Apply the preceding inequality with $(A_0,\lambda)=(1,B)$ to the zeros and with $(A_0,\lambda)=(A,B)$ to the points $w_i$. Writing $P=\prod_i|p(w_i)|$, the paired-product estimate becomes
\[
 P^2\le t^{2n}
   \bigl[A^{n(n-1)/2}J^n\bigr]
   \bigl[A^{3n(n-1)/2}J^n\bigr]
   =\bigl[tA^{n-1}J\bigr]^{2n}.
\]
Since $t=A-1$,
\[
 tA^{n-1}\frac{1-A^{-n}}{1-A^{-1}}=A^n-1.
\]
Both sides of the desired inequality are nonnegative. Taking square roots proves the proposition.
\end{proof}

\begin{corollary}[Exact radius for centers at zeros]\label{cor:baseline}
For every $p\in\Pn$, some zero $z_j$ satisfies
\[
 B(z_j,T_n)\subset\Om_p,\qquad T_n=2^{1/n}-1.
\]
Consequently, $R_n\ge T_n>(\log2)/n$. If the center is required to be a zero, then $T_n$ is the largest possible universal radius.
\end{corollary}
\begin{proof}
If every such disk failed, choose $w_i\in B(z_i,T_n)$ with $|p(w_i)|\ge1$. Since the configuration is finite,
$t=\max_i|w_i-z_i|<T_n$. Proposition~\ref{prop:matched} gives the contradiction
\[
 1\le\prod_i|p(w_i)|\le\bigl((1+t)^n-1\bigr)^n<1.
\]
For $p(z)=z^n-1$, the point $2^{1/n}z_j$ has $|p|=1$ and lies at distance $T_n$ from $z_j$. Thus every larger disk centered at $z_j$ contains a point outside $\Om_p$. Finally, $e^x>1+x$ for $x>0$ gives the strict logarithmic bound.
\end{proof}

This obstruction concerns the restriction on the center. Figure~\ref{fig:root} displays the certified zero-centered disk for the model with six zeros. The remaining sections allow the center to vary and determine the sharp asymptotic coefficient.
\begin{figure}[htbp]
\centering\begin{tikzpicture}[scale=5,>=Latex]
\fill[blue!5] (0,0) -- plot[domain=-14.99:14.99,samples=181] ({(2*cos(6*\x))^(1/6)*cos(\x)},{(2*cos(6*\x))^(1/6)*sin(\x)}) -- cycle;
\draw[blue!60!black,thick] (0,0) -- plot[domain=-14.99:14.99,samples=181] ({(2*cos(6*\x))^(1/6)*cos(\x)},{(2*cos(6*\x))^(1/6)*sin(\x)}) -- cycle;
\draw[->,gray] (-.04,0)--(1.23,0) node[right]{$\Re z$};
\draw[blue,thick] (1,0) circle[radius={2^(1/6)-1}];
\fill (1,0) circle[radius=.009] node[below=7pt]{$1$};
\draw[blue] (1,0)--(1,{2^(1/6)-1}) node[above=2pt]{$T_6$};
\node[below left] at (0,0){$0$};
\end{tikzpicture}
\caption{The principal lobe of $\{|z^6-1|<1\}$ and the open disk $B(1,T_6)$, where $T_6=2^{1/6}-1$. The disk inclusion is the zero-centered result in Corollary~\ref{cor:baseline}; its rightmost boundary point satisfies $|z^6-1|=1$. The boundary curves are schematic renderings of the indicated exact sets.}\label{fig:root}
\end{figure}
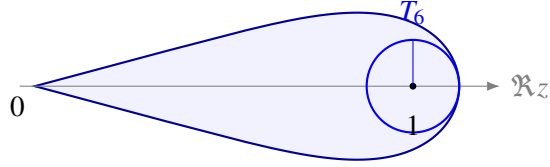

\section{Small inradius and exponential radial stability}\label{sec:radial}
Fix $C>0$. We consider polynomials satisfying
\begin{equation}\label{eq:small}
 p\in\Pn,\qquad \rho(p)\le q:=C/n.
\end{equation}
Whenever a conclusion is asserted for sufficiently large $n$, its threshold depends only on $C$, not on the configuration of zeros.

\begin{lemma}[Variance disk]\label{lem:variance}
Let $\mu=n^{-1}\sum_jz_j$. Then
\begin{equation}\label{eq:variance}
 \rho(p)^2\ge\frac1n\sum_j(1-|z_j|^2)+|\mu|^2.
\end{equation}
In particular, under~\eqref{eq:small}, if $C^2<n$, then
\begin{equation}\label{eq:zero-free}
 |z_j|\ge r_0:=\sqrt{1-C^2/n}\qquad(1\le j\le n).
\end{equation}
\end{lemma}
\begin{proof}
The variance identity is
\[
 \sum_j|z-z_j|^2=n|z-\mu|^2+\sum_j|z_j|^2-n|\mu|^2.
\]
For nonnegative numbers with sum strictly less than $n$, the arithmetic--geometric mean inequality makes their product strictly less than $1$; if a factor is zero, this conclusion is immediate. Apply this observation to $|z-z_j|^2$. It follows that $\sum_j|z-z_j|^2<n$ implies $|p(z)|^2<1$, hence $|p(z)|<1$.

The number $V=1-n^{-1}\sum_j|z_j|^2+|\mu|^2$ is nonnegative because $|z_j|\le1$. If $|z-\mu|<\sqrt V$, the variance identity gives $\sum_j|z-z_j|^2<n$. Thus $B(\mu,\sqrt V)\subset\Om_p$ and $\rho(p)\ge\sqrt V$.
This proves~\eqref{eq:variance}; if that radius is zero, the inequality follows from $\rho(p)\ge0$. Under~\eqref{eq:small},
\[
 \sum_j(1-|z_j|^2)\le nq^2=C^2/n.
\]
All summands are nonnegative, which gives~\eqref{eq:zero-free}.
\end{proof}

On the zero-free disk $B(0,r_0)$, put
\[
 h(z)=\log|p(z)|,\qquad
 M(s)=\max_{|z|\le s}[-h(z)]\quad(0\le s<r_0).
\]
The function $h$ is harmonic. Since $h(0)\le0$, the function $M$ is nonnegative and nondecreasing.

\begin{lemma}[Local Harnack contraction]\label{lem:contraction}
If $0\le a<b<r_0$ and $d=b-a>q$, then
\begin{equation}\label{eq:contraction}
 M(a)\le\frac{2q}{d+q}M(b).
\end{equation}
\end{lemma}
\begin{proof}
Fix $|z|\le a$ and $0<\eta<d-q$. Since $\rho(p)\le q$, the disk $B(z,q+\eta)$ cannot be contained in $\Om_p$. It therefore contains a point $w$ with $h(w)\ge0$. The enlargement by $\eta$ is needed to include the case $\rho(p)=q$.

The function $U=h+M(b)$ is harmonic on a neighborhood of $\overline B(z,d)$ and nonnegative on that disk. Harnack's inequality yields
\[
 M(b)\le U(w)
 \le\frac{d+|w-z|}{d-|w-z|}U(z)
 \le\frac{d+q+\eta}{d-q-\eta}\,[h(z)+M(b)].
\]
This form remains valid when $U(z)=0$. Rearranging, we obtain
\[
 -h(z)\le\frac{2(q+\eta)}{d+q+\eta}M(b).
\]
Let $\eta\downarrow0$ and then maximize over $|z|\le a$. No convergence of the chosen points $w$ is required.
\end{proof}

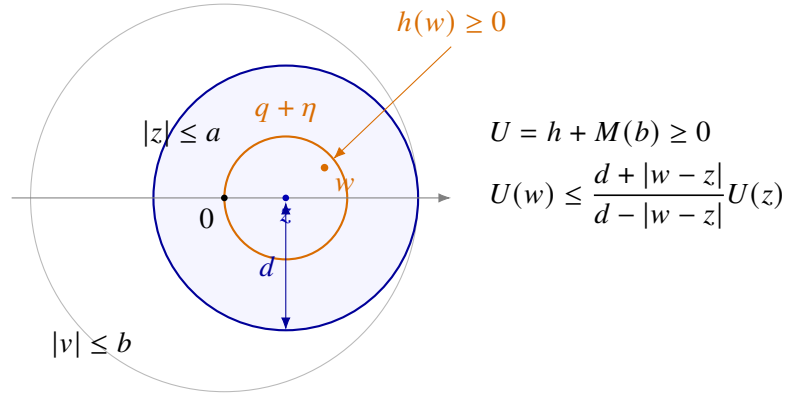
\begin{figure}[htbp]
\centering\begin{tikzpicture}[scale=1.25,>=Latex]
\draw[gray!55] (0,0) circle(2.05);
\draw[gray,dashed] (0,0) circle(.65);
\fill[blue!4] (.65,0) circle(1.4);
\draw[blue!60!black,thick] (.65,0) circle(1.4);
\draw[orange!85!black,thick] (.65,0) circle(.65);
\draw[->,gray] (-2.25,0)--(2.4,0);
\fill (0,0) circle(.035) node[below left]{$0$};
\fill[blue!70!black] (.65,0) circle(.035) node[below]{$z$};
\fill[orange!85!black] (1.06,.32) circle(.04) node[below right]{$w$};
\draw[orange!85!black,->] (2.4,1.6)--(1.15,.41);
\node[orange!85!black,above] at (2.4,1.6){$h(w)\ge0$};
\draw[blue!60!black,<->] (.65,-.04)--(.65,-1.4) node[midway,left]{$d$};
\node[orange!85!black] at (.65,.91){$q+\eta$};
\node at (-1.4,-1.55){$|v|\le b$};
\node at (-.45,.68){$|z|\le a$};
\node[align=left,anchor=west] at (2.7,.25){$U=h+M(b)\ge0$\\[5pt]$U(w)\le\dfrac{d+|w-z|}{d-|w-z|}U(z)$};
\end{tikzpicture}
\caption{The local geometry of Harnack contraction. The larger disk centered at $z$ lies in the zero-free region. The inradius assumption supplies a point $w$ in the smaller disk with $h(w)\ge0$. Harnack's inequality for $h+M(b)$ transfers this information to $z$. The diagram is schematic.}\label{fig:harnack}
\end{figure}

\begin{proposition}[Exponential radial stability]\label{prop:radial}
Assume~\eqref{eq:small} and
\begin{equation}\label{eq:size}
 n\ge256C,\qquad 15n\ge64C^2.
\end{equation}
Then
\begin{equation}\label{eq:radial}
 \Delta:=\sum_j(1-|z_j|)
 \le-\log|p(0)|
 \le16C\,2^{-\lfloor n/(4C)\rfloor}.
\end{equation}
Moreover, with $\ell=\lfloor n/(12C)\rfloor$,
\begin{equation}\label{eq:half-envelope}
 M(1/2)\le16C\,2^{-\ell}.
\end{equation}
\end{proposition}
\begin{proof}
The size assumptions give $q\le1/256$ and $r_0\ge7/8$. We first compare logarithmic factors directly. For $v,w\ne\zeta$, the triangle inequality and $\log u\le u-1$ for $u>0$ give
\[
 \log|w-\zeta|-\log|v-\zeta|
 =\log\frac{|w-\zeta|}{|v-\zeta|}
 \le\frac{|w-v|}{|v-\zeta|}.
\]
If $|v|\le s<r_0$, every denominator $|v-z_j|$ is at least $r_0-s$. Summing over the zeros yields
\[
 h(w)-h(v)\le\frac{n|w-v|}{r_0-s}
\]
whenever neither evaluation point is a zero.

Take $s=3/4$ and choose $0<\eta<r_0-s-q$. The inradius assumption gives $w\in B(v,q+\eta)$ with $h(w)\ge0$; this point is in the zero-free disk. Hence
\[
 -h(v)\le\frac{n(q+\eta)}{r_0-s}.
\]
Letting $\eta\downarrow0$ and taking the maximum gives
\begin{equation}\label{eq:initial-envelope}
 M(3/4)\le\frac C{r_0-3/4}
 \le\frac C{r_0-3/4-q}
 \le\frac{256C}{31}\le16C.
\end{equation}
For the last two inequalities, $r_0-3/4-q\ge31/256>0$ and $C>0$.

In Lemma~\ref{lem:contraction}, take steps of length $d=3q$. The contraction factor is then $1/2$. Set
\[
 m=\left\lfloor\frac n{4C}\right\rfloor,\qquad
 b_j=\frac34-3qj\quad(0\le j\le m).
\]
Since $b_m\ge0$, iteration and monotonicity yield
\[
 -h(0)=M(0)\le M(b_m)\le2^{-m}M(3/4)\le16C\,2^{-m}.
\]
Summing $1-r\le-\log r$ over $r=|z_j|\in(0,1]$ proves~\eqref{eq:radial}. If only $\ell$ steps are taken, then $b_\ell\ge1/2$, and the same argument proves~\eqref{eq:half-envelope}.
\end{proof}

\begin{proposition}[Decay of the low reciprocal moments]\label{prop:moments}
Under~\eqref{eq:small} and~\eqref{eq:size}, define
\[
 A_k=\sum_jz_j^{-k},\qquad K=\left\lfloor\frac n{24C}\right\rfloor.
\]
Then
\begin{equation}\label{eq:moments}
 |A_k|\le96Ck(2/3)^K\qquad(1\le k\le K).
\end{equation}
\end{proposition}
\begin{proof}
On $|z|<r_0$, define the analytic logarithmic potential normalized to vanish at zero by
\begin{equation}\label{eq:potential}
 G(z)=-\sum_j\log(1-z/z_j)
     =\sum_{k\ge1}\frac{A_k}{k}z^k.
\end{equation}
Each logarithm is the branch determined by its power series on the unit disk. Thus
\begin{equation}\label{eq:potential-real}
 e^{-G(z)}=\frac{p(z)}{p(0)},\qquad
 \Re G(z)=h(0)-h(z).
\end{equation}
Since $h(0)\le0$, inequality~\eqref{eq:half-envelope} implies
$\Re G\le E:=16C2^{-\ell}$ on $|z|<1/2$.

The Borel--Carath\'eodory estimate for an analytic function vanishing at zero gives
\[
 |G(z)|\le\frac{2E|z|}{1/2-|z|}\qquad(|z|<1/2).
\]
In particular, $|G(z)|\le6E$ when $|z|=3/8$. Cauchy's derivative estimate and $G^{(k)}(0)=k!A_k/k$ imply
\[
 k!\frac{|A_k|}{k}\le\frac{k!\,6E}{(3/8)^k},\qquad
 |A_k|\le96Ck(8/3)^k2^{-\ell}.
\]
Finally, $\ell\ge2K$ and $k\le K$, so
\[
 (8/3)^k2^{-\ell}\le(8/3)^K2^{-2K}=(2/3)^K.
\]
\end{proof}

Radial stability alone does not control the arguments of the zeros. The reciprocal moments in~\eqref{eq:moments} retain the angular information needed below. We use them directly in an analytic estimate.

\section{Direct control of the analytic logarithm}\label{sec:direct}
We require control of the full analytic logarithm on circles at distance $O(1/n)$ from the unit circle. The normalized logarithmic derivative has a real-part bound independent of the configuration. Removing the small low-order terms leaves a high-order zero, which makes this bound effective.

\begin{lemma}[Half-plane estimate with a high-order zero]\label{lem:high-bc}
Let $R>0$, and let $f$ be holomorphic on $B(0,R)$, with $\Re f\le M$ for some $M>0$. If $f$ vanishes to order at least $m$ at zero, where $m\ge1$ is an integer, then, for $|z|<R$,
\begin{equation}\label{eq:high-bc}
 |f(z)|\le\frac{2M\theta}{1-\theta},
 \qquad \theta=(|z|/R)^m.
\end{equation}
\end{lemma}
\begin{proof}
Set $W=f/(2M-f)$. The real-part hypothesis implies $|2M-f|^2-|f|^2=4M(M-\Re f)\ge0$. The denominator has positive real part, so $W$ is holomorphic and $|W|\le1$. The denominator equals $2M$ at zero, so $W$ also vanishes to order at least $m$. The high-order Schwarz lemma gives $|W(z)|\le(|z|/R)^m$. Solving $f=2MW/(1+W)$ proves~\eqref{eq:high-bc}.
\end{proof}

\begin{lemma}[Radial primitive estimate]\label{lem:primitive}
Let $r>0$, and suppose that $g$ is holomorphic on $B(0,r)$, $g(0)=0$, and
\[
 |zg'(z)|\le A(|z|/r)^m\qquad(|z|<r),
\]
where $A\ge0$ and $m$ is a positive integer. Then
\begin{equation}\label{eq:primitive}
 |g(z)|\le\frac A m(|z|/r)^m.
\end{equation}
\end{lemma}
\begin{proof}
Fix $|z|<r$ and set $f(t)=g(tz)$ for $0\le t\le1$. The radial segment stays in $B(0,r)$, so $f$ is continuous on $[0,1]$ and differentiable on $(0,1)$, with $f'(t)=zg'(tz)$. For $0<t<1$, the hypothesis implies
\[
 t|f'(t)|\le A t^m(|z|/r)^m,
 \qquad |f'(t)|\le A t^{m-1}(|z|/r)^m.
\]
The derivative bound is integrable because $m\ge1$. The norm estimate by the integral of a derivative therefore gives
\[
 |g(z)-g(0)|\le A(|z|/r)^m\int_0^1t^{m-1}\,dt
 =\frac A m(|z|/r)^m.
\]
Use $g(0)=0$ to conclude. This argument includes $z=0$ and uses no division by $t$ at the endpoint.
\end{proof}

\begin{proposition}[Uniform left logarithmic bound]\label{prop:uniform-log}
Fix $C>0$. There is $D=D(C)>0$ such that, for all sufficiently large $n$ and every $p\in\Pn$ with $\rho(p)\le C/n$, the following holds: if $a>1$ and $|z|=e^{-a/n}$, then
\begin{equation}\label{eq:uniform-log}
 |G(z)|\le E_n+D e^{-(a-1)/(24C)},\qquad
 E_n=96CK_n(2/3)^{K_n},\quad
 K_n=\left\lfloor\frac n{24C}\right\rfloor.
\end{equation}
Here $G$ is the analytic logarithm in~\eqref{eq:potential}, continued by the same power series to the disk of radius $\min_j|z_j|$, and $E_n\to0$. The degree threshold is independent of $a$ and of the zeros.
\end{proposition}
\begin{proof}
Write
\[
 b_n=16C\,2^{-\lfloor n/(4C)\rfloor},\quad
 R=1-b_n,\quad K=K_n,\quad m=K+1.
\]
Proposition~\ref{prop:radial} gives $|z_j|\ge R$ for every $j$. For sufficiently large $n$, we have $0<R\le1$. On $|z|<R$, set
\[
 T(z)=\frac{zp'(z)}{np(z)}=-\frac znG'(z).
\]
For $u=z/z_j$, we have $|u|<1$ and
\[
 \Re\frac{-u}{1-u}
 =\frac12-\frac{1-|u|^2}{2|1-u|^2}\le\frac12.
\]
Averaging this identity over the zeros gives
\begin{equation}\label{eq:real-T}
 \Re T(z)\le\frac12\qquad(|z|<R).
\end{equation}

Separate the low-order terms by writing
\[
 L(z)=-\frac1n\sum_{k=1}^K A_kz^k,\qquad H=T-L,
 \qquad P(z)=\sum_{k=1}^K\frac{A_k}{k}z^k.
\]
Then $H$ vanishes to order at least $m$ at zero and $z(G-P)'=-nH$. By~\eqref{eq:moments}, for $|z|\le1$,
\begin{equation}\label{eq:low-budgets}
 |L(z)|\le e_n:=\frac{96CK(K+1)}n(2/3)^K,
 \qquad |P(z)|\le E_n.
\end{equation}
The factor $K(K+1)$ in $e_n$ is a convenient non-sharp bound: the formal estimate bounds each of the $K+1$ summands indexed by $0\le k\le K$ by its value at $K$. We retain this slack throughout. Thus $\Re H\le\tfrac12+e_n$. Lemma~\ref{lem:high-bc} implies
\begin{equation}\label{eq:high-H}
 |H(z)|\le(1+2e_n)\frac{\theta(z)}{1-\theta(z)},
 \qquad\theta(z)=(|z|/R)^m.
\end{equation}

Take $r=e^{-1/n}$. Since $nb_n\to0$, eventually $r<R$. Put
\[
 \sigma_n=(r/R)^m,\qquad
 B_n=2e_n+(1+2e_n)\frac{\sigma_n}{1-\sigma_n}.
\]
On $|z|=r$, equation~\eqref{eq:high-H} and $|L|\le e_n$ first give
\[
 |T|\le e_n+(1+2e_n)\frac{\sigma_n}{1-\sigma_n},
 \qquad |H|=|T-L|\le B_n.
\]
The maximum modulus principle first gives $|H|\le B_n$ on the disk of radius $r$. Since $H$ vanishes to order at least $m$, the high-order Schwarz lemma then gives
\[
 |H(z)|\le B_n(|z|/r)^m\qquad(|z|\le r).
\]
Lemma~\ref{lem:primitive}, applied to $G-P$, now yields
\begin{equation}\label{eq:prelimit-log}
 |G(z)|\le E_n+\frac{nB_n}{m}(|z|/r)^m\qquad(|z|<r).
\end{equation}

We next verify that the constants can be chosen uniformly. Since $K/n\to1/(24C)$,
\[
 E_n\to0,\qquad e_n\to0,\qquad \frac nm\to24C.
\]
Also, $-\log(1-b_n)\le b_n/(1-b_n)$ gives $n\log R\to0$. Hence
\[
 \log\sigma_n=m\left(-\frac1n-\log R\right)
 \longrightarrow-\frac1{24C}.
\]
In particular, $\sigma_n\to e^{-1/(24C)}<1$, so $nB_n/m$ is eventually bounded by a positive constant $D$ depending only on $C$. None of these quantities depends on the zeros.

Finally, if $|z|=e^{-a/n}$ with $a>1$, then
\[
 (|z|/r)^m=e^{-m(a-1)/n}
 \le e^{-(a-1)/(24C)},
\]
because $m=K+1>n/(24C)$. Substitution in~\eqref{eq:prelimit-log} proves the proposition.
\end{proof}

The real-part bound~\eqref{eq:real-T} allows the reciprocal moments to enter the analytic estimate directly. The geometric input at this stage is the lower bound $R$ on the moduli of the zeros; their angular dependence is retained in the coefficients $A_k$.

\section{The first rescaling near the unit circle}\label{sec:first}
Fix $C>0$, and let $p_n\in\Pn$ satisfy $\rho(p_n)\le C/n$. The degrees may range over any sequence tending to infinity. We discard finitely many terms so that all estimates from the preceding section apply.

Choose the first zero to fix the direction, and normalize the constant term by setting
\begin{equation}\label{eq:canonical}
 \xi_n=\frac{z_{1,n}}{|z_{1,n}|},\qquad
 \omega_n=\frac{|p_n(0)|}{p_n(0)},\qquad
 Q_n(w)=\omega_np_n(\xi_ne^{w/n}).
\end{equation}
Radial stability makes all denominators nonzero. Write
\[
 \beta_n=\omega_np_n(0)=|p_n(0)|.
\]
By~\eqref{eq:radial}, $0<\beta_n\le1$ and $\beta_n\to1$. This normalization fixes both the angular coordinate and the limiting value at the left end.

\begin{lemma}[Exact finite-product comparisons]\label{lem:cosh}
Fix $t_{j,n}=n\arg(z_{j,n}/\xi_n)$, using the principal argument; this specifies the angle representative in the following identity. Write
\[
 z_{j,n}/\xi_n=r_{j,n}e^{it_{j,n}/n},\qquad
 s_{j,n}=n\log r_{j,n},\qquad
 S_n(w)=e^{-w/2}Q_n(w).
\]
There are numbers $\eta_n\ge0$ tending to zero such that $-\eta_n\le s_{j,n}\le0$. For all $x,y\in\R$,
\begin{equation}\label{eq:cosh}
 |S_n(x+iy)|^2=\prod_{j=1}^n
 \left\{2r_{j,n}\left[\cosh\frac{x-s_{j,n}}n-\cos\frac{y-t_{j,n}}n\right]\right\}.
\end{equation}
Consequently, if $a\ge0$ and $|x|\le a$, then
\begin{equation}\label{eq:shifted-strip}
 |S_n(x+iy)|\le|S_n(-a-2\eta_n+iy)|.
\end{equation}
If $x\ge0$, then also
\begin{equation}\label{eq:shifted-reflection}
 |S_n(-x+iy)|\le|S_n(x+iy)|
 \le|S_n(-x-2\eta_n+iy)|.
\end{equation}
\end{lemma}
\begin{proof}
With $b_n$ as in the preceding section, take $\eta_n=nb_n/(1-b_n)$. Since $r_{j,n}\ge1-b_n$,
\[
 0\le-n\log r_{j,n}\le-n\log(1-b_n)\le\eta_n\to0.
\]
For each factor, direct calculation gives
\[
 |e^{(x+iy)/n}-r_{j,n}e^{it_{j,n}/n}|^2
 =2r_{j,n}e^{x/n}
 \left[\cosh\frac{x-s_{j,n}}n-\cos\frac{y-t_{j,n}}n\right].
\]
After multiplication, $|e^{-w/2}|^2$ cancels the factor $e^x$, proving~\eqref{eq:cosh}. Each bracket is nonnegative and nondecreasing as a function of $|x-s_{j,n}|$.

For $s\in[-\eta_n,0]$ and $|x|\le a$,
\[
 |x-s|\le a+\eta_n\le|-a-2\eta_n-s|.
\]
Comparison factor by factor proves~\eqref{eq:shifted-strip}. For $x\ge0$,
\[
 |-x-s|\le|x-s|\le|-x-2\eta_n-s|,
\]
which similarly proves~\eqref{eq:shifted-reflection}.
\end{proof}

\begin{proposition}[The first limit and its left asymptotics]\label{prop:first-limit}
The sequence $Q_n$ has a subsequence converging locally uniformly on $\C$ to an entire function $F$ with
\begin{equation}\label{eq:F-zero}
 F(0)=0.
\end{equation}
More precisely, with $D=D(C)$ chosen in Proposition~\ref{prop:uniform-log}, every subsequential limit of the same normalized sequence satisfies
\begin{equation}\label{eq:limit-strip-explicit}
 |F(x+iy)|\le \mathcal B_C(R)
 :=e^{R/2}e^{(\max\{R,2\}+2)/2}e^{1+D}
 \qquad(|x|\le R,\ y\in\R,\ R\ge0).
\end{equation}
The bound depends on $C$ and $R$, and is independent of the chosen convergent subsequence. Moreover, there are constants $K_0,\delta>0$ such that
\begin{equation}\label{eq:left-amplitude}
 |F(-a+iy)-1|\le K_0e^{-\delta a}
 \qquad(a\ge2,\ y\in\R).
\end{equation}
One may take $\delta=1/(24C)$.
\end{proposition}
\begin{proof}
We first obtain uniform bounds before taking a limit. Set $z=\xi_ne^{(-a+iy)/n}$. Then $|z|=e^{-a/n}$, and~\eqref{eq:potential-real} gives
\begin{equation}\label{eq:finite-factor}
 Q_n(-a+iy)=\beta_n e^{-G_n(z)}.
\end{equation}
For sufficiently large $n$, we have $E_n\le1$. Proposition~\ref{prop:uniform-log} therefore yields
\begin{equation}\label{eq:finite-left-bound}
 |Q_n(-a+iy)|\le e^{1+D}\qquad(a\ge2,\ y\in\R).
\end{equation}

Fix $R\ge0$, and put $a=\max\{R,2\}$. Increase the degree threshold so that $\eta_n\le1$. If $|x|\le R$, Lemma~\ref{lem:cosh} gives
\begin{align}
 |Q_n(x+iy)|
 &\le e^{x/2}|S_n(-a-2\eta_n+iy)|\notag\\
 &\le e^{R/2}e^{(a+2)/2}e^{1+D}.
 \label{eq:finite-strip-bound}
\end{align}
The bound is independent of $n$ and $y$. The remaining finitely many entire functions are bounded on every compact set, so $\{Q_n\}$ is locally bounded. Montel's theorem gives a subsequence, still denoted by $Q_n$, converging locally uniformly to an entire function $F$. Taking a pointwise limit in~\eqref{eq:finite-strip-bound} proves~\eqref{eq:limit-strip-explicit}. The estimate preceding the limit is independent of the subsequence, so the same bound holds for each subsequential limit.

Let $s_n=n\log|z_{1,n}|$. The normalization~\eqref{eq:canonical} gives $Q_n(s_n)=0$, while $|s_n|\le\eta_n\to0$. Local uniform convergence on a fixed disk containing these points, together with continuity of $F$, gives $F(0)=0$.

Finally, fix $a\ge2$ and $y\in\R$, and write
\[
 B(a)=De^{-(a-1)/(24C)}.
\]
Using~\eqref{eq:finite-factor}, $\beta_n\le1$, and $|e^u-1|\le|u|e^{|u|}$, we find
\[
 |Q_n(-a+iy)-\beta_n|
 \le(E_n+B(a))e^{E_n+B(a)}.
\]
Pass to the limit along the same subsequence to obtain
\[
 |F(-a+iy)-1|\le B(a)e^{B(a)}.
\]
The point $y$ was arbitrary and the bound does not depend on it, so the estimate holds uniformly on the entire line. This argument does not interchange a supremum and a limit. Since $B(a)\le D$, the choice
$K_0=De^D e^{1/(24C)}$ proves~\eqref{eq:left-amplitude}.
\end{proof}

\begin{proposition}[Exact reflection in the limit]\label{prop:reflection}
The limit function satisfies
\begin{equation}\label{eq:reflection}
 |F(x+iy)|=e^x|F(-x+iy)|\qquad(x,y\in\R).
\end{equation}
\end{proposition}
\begin{proof}
Set $S(w)=e^{-w/2}F(w)$. Then $S_n\to S$ locally uniformly. Fix $x\ge0$ and $y\in\R$. The moving point $-x-2\eta_n+iy$ tends to $-x+iy$ and eventually lies in a fixed compact disk. Passing to the limit in~\eqref{eq:shifted-reflection} gives
\[
 |S(-x+iy)|\le|S(x+iy)|\le|S(-x+iy)|.
\]
Thus $|S(x+iy)|=|S(-x+iy)|$. Replacing $x$ by $-x$ covers negative $x$. Multiplication by the exponential factors gives~\eqref{eq:reflection}.
\end{proof}

Equation~\eqref{eq:F-zero} shows that $F$ is not identically equal to $1$. The exact reflection identity is a property of the limit. At finite degree, the shifted comparisons in Lemma~\ref{lem:cosh} are essential to the compactness argument.

\begin{lemma}[Transfer of strict sublevel disks]\label{lem:transfer}
Suppose that $Q_n$ is given by~\eqref{eq:canonical} and $Q_n\to F$ locally uniformly. If $r\ge0$ and
$\overline B(c,r)\subset\{w:|F(w)|<1\}$, then
\[
 \liminf_{n\to\infty}n\rho(p_n)\ge r.
\]
\end{lemma}
\begin{proof}
On each fixed compact set, for sufficiently large $n$, define
\[
 \psi_n(w)=n\log(1+w/n),
\]
using the branch of the logarithm given by its power series. For $2|w|<n$, the logarithm remainder estimate gives
\[
 |\psi_n(w)-w|\le|w|^2/n.
\]
Hence $\psi_n\to\mathrm{id}$ locally uniformly, and the exact identity
\[
 Q_n(\psi_n(w))=\omega_np_n\bigl(\xi_n(1+w/n)\bigr)
\]
holds. The images of a fixed compact set under $\psi_n$ eventually lie in another fixed compact set. Local uniform convergence of $Q_n$ and uniform continuity of $F$ on that larger set therefore show that the right-hand side converges locally uniformly to $F(w)$.

Compactness of $\overline B(c,r)$ gives a number $\tau<1$ with $|F|\le\tau$ on the disk. For sufficiently large $n$, the affinely rescaled polynomial also has modulus strictly less than $1$ throughout that disk. The map $w\mapsto\xi_n(1+w/n)$ sends it to a closed disk of radius $r/n$. Since $|\omega_n|=|\xi_n|=1$, we obtain $\rho(p_n)\ge r/n$. For $r=0$, the conclusion follows directly from $\rho(p_n)\ge0$.
\end{proof}

\section{The second rescaling: an exponential tangent}\label{sec:tangent}
We now work with the entire function obtained in the preceding section. The argument below shows how the factor $e^x$ in the reflection identity determines the limiting disk radius.

\begin{lemma}[The left slope of a convex function]\label{lem:slope}
Let $\phi:\R\to\R$ be a finite convex function, and let $\delta>0$. Suppose that
\[
 \phi(x)\le\delta x+O(1)\quad(x\to-\infty),\qquad
 \phi(x)\le x+O(1)\quad(x\to+\infty).
\]
There is $\lambda\in[\delta,1]$ such that, for every fixed $t\in\R$,
\begin{equation}\label{eq:increment}
 \phi(x+t)-\phi(x)\longrightarrow\lambda t
 \qquad(x\to-\infty).
\end{equation}
\end{lemma}
\begin{proof}
Fix an anchor $x_0\in\R$, and write
\[
 a(x)=\frac{\phi(x_0)-\phi(x)}{x_0-x}\quad(x<x_0),\qquad
 \widetilde a(x)=a(\min\{x,x_0-1\}).
\]
Convexity makes $\widetilde a$ nondecreasing. Choose constants $X,B$ so that $\phi(x)\le\delta x+B$ for $x\le X$. If
\[
 T=\min\{X,x_0-1,\phi(x_0)-B-(\delta-1)x_0\},
\]
then, for $x\le T$,
\[
 a(x)\ge\frac{\phi(x_0)-\delta x-B}{x_0-x}\ge\delta-1.
\]
Monotonicity gives the same lower bound for $\widetilde a(x)$ when $x\ge T$. Thus $\widetilde a$ has a finite limit $\lambda$ at $-\infty$, and so does $a$.

We check the two bounds on $\lambda$. If $\lambda<\delta$, choose $\lambda<\mu<\delta$. Convergence gives $a(x)<\mu$ for all sufficiently negative $x$. But the left affine bound gives $a(x)>\mu$ whenever, in addition,
\[
 x<\frac{\phi(x_0)-B-\mu x_0}{\delta-\mu},
\]
a contradiction. Hence $\delta\le\lambda$.

For the upper bound, write $\phi(t)\le t+B_+$ for $t\ge X_+$. If $s=a(x)>1$ at some $x<x_0$, choose
\[
 t>\max\left\{X_+,x_0,
      \frac{B_+-\phi(x_0)+s x_0}{s-1}\right\}.
\]
The right affine bound gives $(\phi(t)-\phi(x_0))/(t-x_0)<s$, whereas convexity gives the opposite weak inequality. Thus $a(x)\le1$ for $x<x_0$, and $\lambda\le1$.

For $t>0$ and $x+t<x_0$, convexity now yields the mirror-secant squeeze
\[
 \frac{\phi(x)-\phi(2x-x_0)}{x_0-x}
 \le\frac{\phi(x+t)-\phi(x)}t\le a(x).
\]
The left member is exactly $2a(2x-x_0)-a(x)$, which tends to $\lambda$. The middle member therefore tends to $\lambda$. Multiplying by $t$ proves~\eqref{eq:increment} for $t>0$. For $t<0$, apply the positive-increment result at $x+t$ with increment $-t$ and negate; the case $t=0$ is immediate.
\end{proof}

\Needspace{13\baselineskip}
\begin{theorem}[Exponential tangents and sublevel disks]\label{thm:tangent}
Let $F$ be an entire function not identically equal to $1$, satisfying the following conditions:
\begin{enumerate}[label=\textup{(\roman*)}]
\item For each $R>0$, the function $F$ is bounded on the full strip $|\Re w|\le R$.
\item There are $K_0,\delta>0$ such that $|F(x+iy)-1|\le K_0e^{\delta x}$ for all sufficiently negative $x$ and all $y\in\R$.
\item For all $x,y\in\R$, one has $|F(x+iy)|=e^x|F(-x+iy)|$.
\end{enumerate}
Then, for every $0\le r<\pi/2$, there exists $c\in\C$ such that
\begin{equation}\label{eq:limit-disks}
 \overline B(c,r)\subset\{w:|F(w)|<1\}.
\end{equation}
\end{theorem}
\begin{proof}
Put
\[
 h=F-1,\qquad M(x)=\sup_{y\in\R}|h(x+iy)|,
 \qquad\phi(x)=\log M(x).
\]
Condition~(i) gives $M(x)<\infty$. If $M(x)=0$ for some $x$, then $h(x+iy)=0$ for every real $y$. These zeros accumulate at the finite point $x$: take $x+it$ with real $t\ne0$ tending to zero. The identity theorem therefore gives $h\equiv0$, hence $F\equiv1$, a contradiction. Hence $M(x)>0$, and $\phi$ is finite everywhere.

\emph{Three-lines convexity.}
Fix $x_1<x_2$. The function $h$ is bounded on the full closed strip $x_1\le\Re w\le x_2$, so Hadamard's three-lines theorem gives
\[
 M(x)\le M(x_1)^{(x_2-x)/(x_2-x_1)}
              M(x_2)^{(x-x_1)/(x_2-x_1)}
 \quad(x_1<x<x_2).
\]
Taking logarithms shows that $\phi$ is convex. Condition~(ii) gives $\phi(x)\le\delta x+O(1)$ at the left end. For sufficiently large positive $x$, conditions~(ii) and~(iii) imply
\[
 |h(x+iy)|\le|F(x+iy)|+1
 \le e^x(1+K_0e^{-\delta x})+1\le(K_0+2)e^x.
\]
Thus $\phi(x)\le x+O(1)$ at the right end. Lemma~\ref{lem:slope} gives a number
\begin{equation}\label{eq:lambda}
 0<\delta\le\lambda\le1
\end{equation}
for which~\eqref{eq:increment} holds.

\emph{A nonzero tangent.}
Choose $x_j=-(j+1)$ and $\epsilon_j=1/(j+2)$ for $j\ge0$. Although the supremum defining $M(x_j)$ need not be attained, we may select $y_j\in\R$ such that
\[
 (1-\epsilon_j)M(x_j)\le|h(x_j+iy_j)|\le M(x_j),
 \qquad 0<\epsilon_j<1,\quad\epsilon_j\to0.
\]
Set $m_j=M(x_j)$ and
\begin{equation}\label{eq:tangent-sequence}
 H_j(w)=\frac{h(x_j+iy_j+w)}{m_j}.
\end{equation}
If $u=\Re w$, then
\begin{equation}\label{eq:tangent-majorant}
 |H_j(w)|\le\exp\{\phi(x_j+u)-\phi(x_j)\}.
\end{equation}
To check local boundedness, fix $R>0$. The convex function
$u\mapsto\phi(x_j+u)-\phi(x_j)$ is bounded above on $[-R,R]$ by the larger of its two endpoint values. By~\eqref{eq:increment}, these endpoint values tend to $-\lambda R$ and $\lambda R$, respectively. Thus $\{H_j\}$ is locally bounded. After passing to a subsequence, Montel's theorem gives local uniform convergence $H_j\to H$ on $\C$.

At zero, the choice of $y_j$ implies $|H(0)|=1$. Passing to the limit in~\eqref{eq:tangent-majorant} gives
\[
 |H(w)|\le e^{\lambda\Re w}\qquad(w\in\C).
\]
The entire function $e^{-\lambda w}H(w)$ has modulus at most $1$ and attains that value at zero. The maximum modulus principle yields
\begin{equation}\label{eq:pure-exp}
 H(w)=ce^{\lambda w},\qquad |c|=1.
\end{equation}

\emph{A strict negative margin.}
Since $\lambda>0$, choose $\tau\in\R$ so that $ce^{i\lambda\tau}=-1$. Fix $0\le r<\pi/2$. From $\lambda\le1$, for all $|z|\le r$ we have
\begin{equation}\label{eq:negative-margin}
 \Re H(i\tau+z)
 =-e^{\lambda\Re z}\cos(\lambda\Im z)
 \le-e^{-\lambda r}\cos(\lambda r)=:-\eta_r<0.
\end{equation}
Local uniform convergence implies that, for sufficiently large $j$, throughout $|z|\le r$,
\[
 \Re H_j(i\tau+z)\le-\eta_r/2,\qquad
 |H_j(i\tau+z)|\le B_r,
\]
where $B_r$ is independent of $j$. Condition~(ii) gives $m_j\to0$, and $m_j>0$. Using the exact identity~\eqref{eq:tangent-sequence}, we obtain
\begin{align*}
 |F(x_j+i(y_j+\tau)+z)|^2
 &=|1+m_jH_j(i\tau+z)|^2\\
 &\le1-\eta_rm_j+B_r^2m_j^2<1
\end{align*}
for all sufficiently large $j$, uniformly on $|z|\le r$. This proves~\eqref{eq:limit-disks}.
\end{proof}

\begin{figure}[htbp]
\centering\begin{tikzpicture}[>=Latex,scale=1.15]
\fill[blue!5] (-2.5,-1.57) rectangle (2.3,1.57);
\draw[blue!60!black,dashed] (-2.5,1.57)--(2.3,1.57);
\draw[blue!60!black,dashed] (-2.5,-1.57)--(2.3,-1.57);
\draw[->,gray] (-2.65,0)--(2.5,0) node[right]{$\Re z$};
\draw[->,gray] (0,-1.9)--(0,2.0) node[above]{$\Im z$};
\draw[orange!85!black,thick,fill=orange!8] (0,0) circle(1.15);
\draw[orange!85!black,->] (0,0)--(.813,.813) node[midway,below]{$r$};
\fill (0,0) circle(.03) node[below left]{$0$};
\node[anchor=west] at (2.45,1.57){$\pi/2$};
\node[anchor=west] at (2.45,-1.57){$-\pi/2$};
\node[align=left,anchor=west] at (3.6,.25){$\Re(-e^z)=-e^x\cos y<0$\\[7pt]$|1+mH_j|^2$\\[3pt]$\quad\le1-\eta_rm+B_r^2m^2$};
\end{tikzpicture}
\caption{The geometry of the exponential tangent, shown for $\lambda=1$. The real part of $-e^z$ is negative in $|\Im z|<\pi/2$. Every fixed closed disk $|z|\le r<\pi/2$ stays a positive distance from the boundary lines. The resulting uniform negative margin controls the quadratic error in $|1+mH_j|^2$ as $m\downarrow0$.}\label{fig:exponential}
\end{figure}
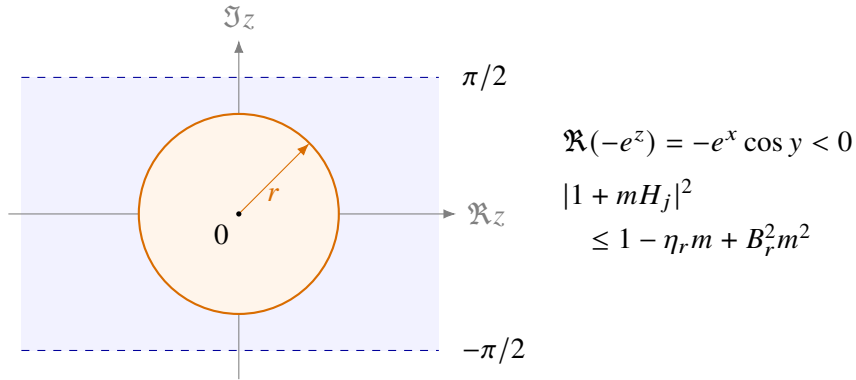

The positive constant $\eta_r$ in~\eqref{eq:negative-margin} is essential. A nonpositive real part alone would not control the term $m_j^2|H_j|^2$. The estimate must also be uniform on the entire closed disk to permit the transfer in Lemma~\ref{lem:transfer}.

\section{The sharp asymptotic constant}\label{sec:conclusion}
\begin{proof}[Proof of Theorem~\ref{thm:main}]
We first establish the uniform lower bound. Fix $\varepsilon>0$. If the assertion failed, there would be increasing degrees $n_j\to\infty$ and polynomials $p_j\in\mathcal P_{n_j}$ such that
\begin{equation}\label{eq:bad}
 n_j\rho(p_j)<\pi/2-\varepsilon.
\end{equation}
If $\pi/2-\varepsilon\le0$, this contradicts $\rho(p_j)\ge0$, so we may assume that the right-hand side is positive. Take $C=\pi/2$. The sequence satisfies $\rho(p_j)\le C/n_j$, and its scaled inradii lie in $[0,C]$. Passing to a subsequence, we may assume
\[
 n_j\rho(p_j)\longrightarrow L\le\pi/2-\varepsilon.
\]
Proposition~\ref{prop:first-limit} provides a further subsequence with a locally uniform limit $F$. Equations~\eqref{eq:F-zero} and~\eqref{eq:left-amplitude}, together with Proposition~\ref{prop:reflection}, show that $F$ satisfies all hypotheses of Theorem~\ref{thm:tangent}. Thus, for every $0\le r<\pi/2$, the strict sublevel set of $F$ contains a closed disk of radius $r$. Lemma~\ref{lem:transfer} gives $L\ge r$. Letting $r\uparrow\pi/2$ yields $L\ge\pi/2$, a contradiction.

It follows that, for all sufficiently large $n$, every $p\in\Pn$ satisfies $\rho(p)\ge(\pi/2-\varepsilon)/n$. Taking the infimum over all configurations gives
\begin{equation}\label{eq:lower-main}
 \liminf_{n\to\infty}nR_n\ge\pi/2.
\end{equation}
This argument proves uniformity directly and does not require an extremizing polynomial.

For the upper bound, consider $p_n(z)=z^n-1$, and put
\[
 \beta_n=\frac\pi{2n},\qquad \alpha_n=\frac\pi{2(n+1)},\qquad
 q_n=(2\sin\alpha_n)^{1/n},\qquad H_n=q_n\sin\alpha_n.
\]
For $z=re^{i\theta}$ with $r>0$, direct expansion gives
\begin{equation}\label{eq:model-polar}
 |z^n-1|<1\quad\Longleftrightarrow\quad r^n<2\cos(n\theta).
\end{equation}
We show that each inscribed disk has radius at most $H_n$.

\emph{Selecting one lobe.}
On a disk contained in $\{|z^n-1|<1\}$, we have $z\ne0$ and $\Re(z^n)>0$. Use the principal branch to define
\[
 v(z)=(z^n)^{1/n},\qquad u(z)=v(z)/z.
\]
Then $u$ is continuous, $u(z)^n=1$, and $u(z)z=v(z)$ lies in the principal lobe $|\arg v|<\beta_n$. Since the disk is connected and the set of $n$th roots of unity is finite, $u$ is constant on the disk. Multiplication by this constant rotates the entire disk into the principal lobe and preserves its radius.

\emph{The vertical width of the principal lobe.}
For $0\le\theta\le\beta_n$, let $b=\pi/2-n\theta$. The numbers $b,\theta$ lie in $[0,\pi]$, and
\[
 \frac{b+n\theta}{n+1}=\alpha_n.
\]
Concavity of sine followed by the weighted arithmetic--geometric mean inequality gives
\[
 (\sin b)^{1/(n+1)}(\sin\theta)^{n/(n+1)}
 \le\frac{\sin b+n\sin\theta}{n+1}\le\sin\alpha_n.
\]
Raising to the power $n+1$ and using $\sin b=\cos(n\theta)$ yields
\[
 \cos(n\theta)\sin^n\theta\le\sin^{n+1}\alpha_n.
\]
For a point $z=re^{i\theta}$ in the principal lobe with $\theta\ne0$, equation~\eqref{eq:model-polar}, applied to $|\theta|$, implies
\[
 |\Im z|^n=r^n\sin^n|\theta|
 <2\cos(n|\theta|)\sin^n|\theta|
 \le2\sin^{n+1}\alpha_n=H_n^n.
\]
For $\theta=0$, $\Im z=0$ and $H_n>0$. Thus the principal lobe is contained in $|\Im z|<H_n$.

If a disk $B(c,s)$ in this strip had $s>H_n$, choose $H_n<t<s$. Both $c+it$ and $c-it$ would belong to the strip, giving $\Im c+t<H_n$ and $t-\Im c<H_n$. Adding gives $t<H_n$, a contradiction. Consequently
\[
 R_n\le\rho(z^n-1)\le H_n.
\]

\emph{Taking the limit.}
We have $\alpha_n\to0$ and $n\alpha_n\to\pi/2$, so $n\sin\alpha_n\to\pi/2$. Writing $\operatorname{sinc}t=\sin t/t$ for $t\ne0$ and $\operatorname{sinc}0=1$, the identity
\[
 2\sin\alpha_n=\frac\pi{n+1}\operatorname{sinc}\alpha_n
\]
gives
\[
 \log q_n=\frac{\log\pi-\log(n+1)+
                  \log(\operatorname{sinc}\alpha_n)}n\longrightarrow0.
\]
Here $\log(n+1)/n\to0$ and $\operatorname{sinc}\alpha_n\to1$. Hence $q_n\to1$, $nH_n\to\pi/2$, and
\[
 \limsup_{n\to\infty}nR_n\le\pi/2.
\]
Together with~\eqref{eq:lower-main}, this proves~\eqref{eq:main}. The same bounds imply $n\rho(z^n-1)\to\pi/2$, which excludes any larger asymptotic universal coefficient.
\end{proof}

The order of the two limiting arguments is fixed throughout the proof. For a prescribed $r<\pi/2$, the second rescaling is carried out within the already constructed entire function $F$, and a closed disk at a finite location is selected. Only then is the degree allowed to tend to infinity to transfer this fixed disk to the original polynomials. No synchronization of the rescaling parameters or change in the sublevel threshold is needed.

The product argument supplies the correct scale. Radial stability and low-order moment decay give the local limit its reflection structure. Three-lines convexity then produces an exponential tangent with exponent at most $1$, turning the width of a half-strip into the sharp inradius coefficient $\pi/2$.

\section*{Declarations}
\noindent\textit{Author and AI contributions.}
Yujian Geng conceived and directed the project, developed an initial proof of Bloom's matched-product inequality, and introduced the reciprocal-smoothing approach, using paired regularization scales related by inversion to retain modulus information in the product estimates. He investigated small perturbations near the zeros and displacements of disk centers as mechanisms for improving inradius bounds, and provided mathematical guidance in the development and refinement of later parts of the proof. Dong Qiu carried out the human mathematical review of the manuscript and provided theoretical guidance. ChatGPT assisted with proof development, exposition, and the generation of Lean~4 code, which was subsequently rechecked using the authors' Canto Prover system. The authors are responsible for the entire mathematical content, its attribution, and the verification of all claims, including material developed with AI assistance. The baseline code derives from Kitamura's publicly released formalization~\cite{Kitamura}; its provenance and license are retained in the archive.

\medskip
\noindent\textit{Definitions and verification scope.}
The formal theorem uses the strict sublevel set and the open Euclidean balls in~\eqref{eq:definitions}; it does not replace the strict sublevel condition by a non-strict one. The internal radius supremum allows nonpositive radii, whose open balls are empty. The file \texttt{PaperDefinitionBridge.lean} proves that this convention gives exactly the positive-radius supremum in~\eqref{eq:definitions} for positive degree, and that the extremal infima agree. No attainment of either extremum is assumed. Closed disks in Lemma~\ref{lem:transfer} and Theorem~\ref{thm:tangent} are compact witnesses strictly inside the sublevel set; they do not alter the definition of inradius. The final statement in the paper's notation is \texttt{main\_positive\_radius\_stepwise} in \texttt{PaperStepDetails.lean}. The accompanying \texttt{STEP\_INDEX.md} records the individual mathematical arguments and their Lean source locations. This correspondence is a mathematical audit, not an automatic verification of English prose.

\medskip
\noindent\textit{Verification resources and reproducibility.}
The complete formalization and the accompanying manuscript are available at
\begin{center}\small\url{https://github.com/xyzjpkkk-max/erdos1039-lean}.\end{center}
The fixed code snapshot used for this paper is commit
\begin{center}\small\texttt{9ca0a0b8c79c166866214dbb519f215086e1e783}.\end{center}
The repository retains this commit under the tag \texttt{code-v1}. Its Lean project is at the root of that snapshot and in the \texttt{lean/} directory of the accompanying paper release. The Lean source files are identical. A Git bundle and SHA-256 manifest provide an additional archive of this version. The code uses Lean~4.29.1 and Mathlib~v4.29.1; the latter is pinned to commit
\begin{center}\small\texttt{5e932f97dd25535344f80f9dd8da3aab83df0fe6}.\end{center}
With Lean installed, the supplied \texttt{verify.ps1} script checks the source manifest, obtains the pinned dependencies and their cache, builds \texttt{PaperReviewAudit}, and checks the recorded axiom reports. The README also gives the individual build commands and a Python checker for the source index. The transitive axiom dependencies are restricted to \texttt{propext}, \texttt{Classical.choice}, and \texttt{Quot.sound}; \texttt{sorryAx} does not occur. Build logs and SHA-256 manifests are included. The reported build uses the Mathlib cache and does not claim a recompilation of all Mathlib sources from scratch.

\medskip
\noindent\textit{Competing interests.}
The authors declare no competing interests.

\section*{Acknowledgements}
This work was supported by the National Natural Science Foundation of China (Nos.~12571489 and~12171065) and the Guangxi Natural Science Foundation (No.~2025GXNSFAA069576).

\Needspace{12\baselineskip}

\end{document}